\documentclass[leqno,12pt]{amsart} 
\usepackage{amssymb}
\usepackage{amsmath}
\usepackage{amsthm}
\usepackage{amscd}
\usepackage{graphicx}
\usepackage[all]{xy}
\usepackage[dvips]{color}
\theoremstyle{plain} 
\newtheorem{theorem}{\indent\sc Theorem}[section]
\newtheorem{lemma}[theorem]{\indent\sc Lemma}
\newtheorem{corollary}[theorem]{\indent\sc Corollary}
\newtheorem{proposition}[theorem]{\indent\sc Proposition}

\theoremstyle{definition} 
\newtheorem{definition}[theorem]{\indent\sc Definition}
\newtheorem{remark}[theorem]{\indent\sc Remark}
\newtheorem{example}[theorem]{\indent\sc Example}

\begin{document}

\title[On minimal singular metrics of certain class of line bundles]
{On minimal singular metrics of certain class of line bundles whose section ring is not finitely generated} 

\author[T. Koike]{Takayuki Koike} 

\subjclass[2010]{ 
32J25; 14C20. 
}
%
\keywords{ 
minimal singular metrics, tubular neighborhoods, Zariski's example. 
}
\address{
Graduate School of Mathematical Sciences, The University of Tokyo \endgraf
3-8-1 Komaba, Meguro-ku, Tokyo, 153-8914 \endgraf
Japan
}
\email{tkoike@ms.u-tokyo.ac.jp}

\maketitle

\begin{abstract}
Our interest is a regularity of a minimal singular metric of a line bundle. 
One main conclusion of our general result in this paper is the existence of smooth Hermitian metrics with semi-positive curvatures on the so-called Zariski's example of a line bundle defined over the blow-up of $\mathbb{P}^2$ at some twelve points. 
This is an example of a line bundle which is nef, big, not semi-ample, and whose section ring is not finitely generated. 
We generalize this result to the higher dimensional case when the stable base locus of a line bundle is a smooth hypersurface with a holomorphic tubular neighborhood. 
\end{abstract}

\section{Introduction}

Our interest is a regularity of a minimal singular metric of a line bundle. 
One main conclusion of our general result in this paper is the existence of smooth Hermitian metrics with semi-positive curvatures on the so-called Zariski's example (\cite[2.3.A]{L1}). 

\begin{theorem}[Example \ref{Zariski}]\label{zariski_theorem}
Let $C\subset \mathbb{P}^2$ be a smooth elliptic curve, 
$\pi\colon X\to \mathbb{P}^2$ the blowing-up at general twelve points $p_1, p_2, \dots, p_{12}\in C$, 
$H$ the pulled back divisor of a line in $\mathbb{P}^2$,  
and let $D$ be the strict transform of $C$.  
Then the line bundle $L=\mathcal{O}_X(H+D)$ is semi-positive (i.e. $L$ admits a smooth Hermitian metric with semi-positive curvature). 
\end{theorem}

This $L$ is nef and big, however has a pathological property that $D\subset \text{Bs}\,|L^{\otimes m}|$ holds for all $m\geq 1$, $|L^{\otimes m}\otimes\mathcal{O}_X(-D)|$ is globally generated for all $m\geq 1$, 
and that the section ring $\bigoplus_{m\geq 0}H^0(X, L^{\otimes m})$ of $L$ is not finitely generated. 
When the twelve points $p_1, p_2, \dots, p_{12}\in C$ is special, the line bundle $L$ is semi-ample and thus it is semi-positive. 
Minimal singular metrics of  a line bundle $L$ are metrics of $L$ with the mildest singularities among singular Hermitian metrics of $L$ whose local weights are plurisubharmonic. 
Minimal singular metrics have been introduced in \cite[1.4] {DPS00} as a (weak) analytic analogue of the Zariski decomposition, and always exist when $L$ is pseudo-effective (\cite[1.5] {DPS00}). 
The main theorem is as follows. 

\begin{theorem}\label{main_theorem}
Let $X$ be a smooth projective variety, 
$D$ a smooth hypersurface of $X$, 
$L$ a pseudo-effective line bundle over $X$, 
and let $h_{\rm min}$ be a minimal singular metric of $L$. 
Assume that $L\otimes\mathcal{O}_X(-D)$ is semi-positive, 
$\mathcal{O}_X(-D)|_D$ is ample, 
$\mathcal{O}_D(-K_D-D|_D)$ is nef and big, 
and that $D$ has a holomorphic tubular neighborhood (i.e. an open neighborhood in $X$ which is biholomorphic to an open neighborhood of the zero section in the normal bundle $N_{D/X}$). 
Then $h_{\rm min}|_D\not\equiv\infty$ holds if and only if $L|_D$ is pseudo-effective, moreover in this case $h_{\rm min}|_D$ is a minimal singular metric of $L|_D$. 
\end{theorem}

One of the typical cases of the situations in Theorem \ref{main_theorem} is when $X$ is a surface and the self-intersection number $(D^2)$ is (sufficiently) negative. 
It is followed by a special case of Grauert's theorem \cite[Satz 7]{G}: 
A smooth compact complex curve $D$ with genus $g$ embedded in a complex surface $X$ has a holomorphic tubular neighborhood 
if $(D^2)<\min\{0, 4-4g\}$ holds. 
Thus, we can apply our main theorem to Zariski's example to obtain Theorem \ref{zariski_theorem}, 
and we also can show the existence of a smooth Hermitian metric with semi-positive curvature for the same type examples introduced by Mumford (\cite[2.3.A]{L1}, or Example \ref{Zariski} here). 
When $(L\otimes\mathcal{O}_X(-D))|_D$ is ample, we can write down more concretely a minimal singular metric of $L$ around $D$ by using equilibrium metrics, which are special minimal singular metrics, of $\mathbb{R}$-line bundles $(L\otimes \mathcal{O}_X(-tD))|_D$ for $0\leq t\leq 1$ (see Theorem \ref{P1_main_theorem} and Remark \ref{main_rmk}).  

Another application of Theorem \ref{main_theorem} we can expect is a concrete description of minimal singular metrics of a pseudo-effective line bundle which is not big. 
It is because, it follows from next corollary, which is a spacial case of Theorem \ref{main_theorem}, that we can apply Bergman kernel construction argument even when $L$ is not big but merely pseudo-effective. 
In more detail, the line bundle $\mathcal{O}_{\mathbb{P}(A\oplus L)}(1)$ in next Theorem \ref{P1_corollary} is big if we chose $A$ as an ample line bundle. 
Thus we can use Bergman kernel construction argument for this line bundle and 
we can study minimal singular metrics of $L$ itself by restricting argument. 

\begin{theorem}\label{P1_corollary}
Let $X$ be a smooth projective variety, $A$ a semi-ample line bundle on $X$, 
and let $L$ be a pseudo-effective line bundle on $X$. 
Then the restriction of a minimal singular metric of $\mathcal{O}_{\mathbb{P}(A\oplus L)}(1)$ on $\mathbb{P}(A\oplus L)$ to the divisor $\mathbb{P}(L)\subset \mathbb{P}(A\oplus L)$ corresponding to the projection $A\oplus L\to L$ gives a minimal singular metric of $L$ on $X$ via the natural identification $(\mathbb{P}(L), \mathcal{O}_{\mathbb{P}(L)}(1))\cong (X, L)$ (see Remark \ref{P1bdl_rmk}). 
\end{theorem}

We can prove this corollary directly by constructing a minimal singular metric of $\mathcal{O}_{\mathbb{P}(A\oplus L)}(1)$ from fixed smooth Hermitian metrics of $A$ and $L$. 
In the proof of Theorem \ref{main_theorem}, we use the assumption on the existence of holomorphic tubular neighborhoods to reduce the situation in Theorem \ref{main_theorem} to that in Theorem \ref{P1_corollary}. 
Since $L$ in Theorem \ref{main_theorem} admits a singular Hermitian metric which is smooth on $X\setminus D$ and may be singular along $D$, all we have to do is to modify this metric around $D$. 
We will replace this metric on the tubular neighborhood of $D$ by the metric constructed in the situation of Theorem \ref{P1_corollary}. 

The organization of the paper is as follows. 
In \S 2, we treat the case when $X$ has a suitable $\mathbb{P}^1$-bundle structure and $L$ is the relative hyperplane bundle, and we prove Theorem \ref{P1_corollary}. 
In \S 3, we prove Theorem \ref{main_theorem}. 
Finally we give some examples in \S 4. 

\vskip3mm
{\bf Acknowledgment. } 
The author would like to thank his supervisor Prof. Shigeharu Takayama whose enormous support and insightful comments were invaluable during the course of his study. 
He also thanks Prof. Shin-ichi Matsumura who gave him invaluable comments. 
He is supported by the Grant-in-Aid for Scientific Research (KAK-
ENHI No.25-2869) and the Grant-in-Aid for JSPS fellows. 
This work is supported by the Program for Leading Graduate
Schools, MEXT, Japan. 

\section{The $\mathbb{P}^1$-bundle case}

In this section, we treat the case when $X$ has a suitable $\mathbb{P}^1$-bundle structure and $L$ is the relative hyperplane bundle. 
Here we give a minimal singular metric of $L$ concretely by using equilibrium metrics of $\mathbb{R}$-line bundles of the base space of $X$. 
First we define the equilibrium metrics for smooth Hermitian metrics on pseudo-effective line bundles. 

\begin{definition}
Let $X$ be a smooth projective variety, 
$L$ a pseudo-effective line bundle over $X$, 
and let $h=e^{-\varphi}$ be a smooth Hermitian metric on $L$. 
We denote by $h_e$ the {\it equilibrium metric}, whose local weight function $\varphi_e$ is defined by 
\[
\varphi_e=\varphi +\sup\hskip-0.5em\ ^*\{\psi\colon X\to \mathbb{R}\cup\{-\infty\}\mid \psi\text{ is a }\varphi\text{-psh function, }\psi\leq 0\}, 
 \] 
where $\sup\hskip-0.5em\ ^*$ stands for the upper semi-continuous regularization of the supremum. 
\end{definition}

Equilibrium metrics are minimal singular metrics (\cite[1.5] {DPS00}). 
Using this notion, we prove the following theorem, a special case of Theorem \ref{main_theorem} (see Remark \ref{P1bdl_rmk} below). 

\begin{theorem}\label{P1_main_theorem}
Let $X$ be a smooth projective variety, $A$ an ample line bundle on $X$, 
and let $L$ be a pseudo-effective line bundle on $X$. 
Let $h_L=e^{-\varphi_L}$ be a smooth Hermitian metric of $L$ and 
let $h_A=e^{-\varphi_A}$ be a smooth Hermitian metric of $A$ satisfying $dd^c\varphi_A>0$. 
Fix a local coordinate system by
$(z, x)\mapsto[zs_A^*(x)+s_L^*(x)]\in \mathbb{P}(A\oplus L)$, 
where $s_A^*$ and $s_L^*$ are local trivializations of $A^{-1}$ and $L^{-1}$, respectively. 
Then the metric of the relative hyperplane line bundle $\mathcal{O}_{\mathbb{P}(A\oplus L)}(1)$ on $\mathbb{P}(A\oplus L)$ 
defined by the local weights 
\[
\tilde{\varphi}(z, x)=\log \max_{t\in [0, 1]}|z|^{2t}e^{(t\varphi_A+(1-t)\varphi_L)_e(x)} 
\]
is a minimal singular metric, where $(t\varphi_A+(1-t)\varphi_L)_e$ is the local weight of the equilibrium metric associated to $h_A^th_L^{1-t}$, which is a smooth Hermitian metric of the ``\,$\mathbb{R}$-line bundle $A^{\otimes t}\otimes L^{\otimes(1-t)}$''. 
\end{theorem}

\begin{remark}\label{P1bdl_rmk}
Here we use the notations of Theorem \ref{P1_main_theorem}. 
Let us denote by $\tilde{X}$ the total space  $\mathbb{P}(A\oplus L)$, by $\pi$ the projection $\tilde{X}\to X$, by $X'$ the subset $\mathbb{P}(L)$ of $\tilde{X}$, and $X''$ the subset $\mathbb{P}(A)$. 
Then $\mathcal{O}_{\tilde{X}}(X')=\mathcal{O}_{\mathbb{P}(A\oplus L)}(1)\otimes\pi^*A^{-1}$ and 
$\mathcal{O}_{\mathbb{P}(A\oplus L)}(1)|_{X'}=\pi^*L|_{X'}$ hold as equalities of line bundles on $\tilde{X}$ and $X'$, respectively. 
Therefore we can regard the restriction of a metric of $\mathcal{O}_{\mathbb{P}(A\oplus L)}(1)$ to $X'$ as 
a metric of $L$, and by regarding $X'\subset\tilde{X}$ as $\mathbb{P}(\mathcal{O}_X)\subset\mathbb{P}(\mathcal{O}_X\oplus(L^{-1}\otimes A))$, 
we can identify $\tilde{X}\setminus X''$ and $X'$ with the total space of the normal bundle $N_{{X'}/\tilde{X}}$, which is isomorphic to the bundle $L\otimes A^{-1}$ via $\pi$, and its zero-section. 
\end{remark}

From now on, we prove Theorem \ref{P1_main_theorem}. 
We denote by $\tilde{X}$ the variety $\mathbb{P}(A\oplus L)$, 
by $\pi\colon \tilde{X}\to X$ the canonical projection mapping, 
and by $\tilde{L}$ the relative plane line bundle $\mathcal{O}_{\mathbb{P}(A\oplus L)}(1)$ on $\tilde{X}$. 
Here we denote by $U$ the domain of definition of $s_A^*$, $s_L^*$ and $x$. 
We also use the smooth Hermitian metric $\tilde{h}_\infty=e^{-\tilde{\varphi}_\infty}$ of $\tilde{L}$, whose local weight is defined as 
$
\tilde{\varphi}_\infty(z, x)=\log \left(|z|^2e^{\varphi_A(x)}+e^{\varphi_L(x)}\right). 
$
To prove Theorem \ref{P1_main_theorem}, it is sufficient to show the following two propositions. 

\begin{proposition}[Plurisubharmonicity of $\tilde{\varphi}$]\label{lem_usc}
The function 
\[
\tilde{\varphi}(z, x)=\log \max_{t\in [0, 1]}|z|^{2t}e^{(t\varphi_a+(1-t)\varphi_L)_e(x)}
\] 
is plurisubharmonic and $\{e^{-\tilde{\varphi}}\}$ glue up to define a singular Hermitian metric of $\tilde{L}$. 
\end{proposition}

\begin{proposition}[Minimal singularity of $\tilde{\varphi}$]\label{lem_ms}
There is a constant $C$ such that $(\tilde{\varphi}_\infty)_e\leq \tilde{\varphi}+C$ holds. 
\end{proposition}

\subsection{Proof of Proposition \ref{lem_usc}}
Since $\log |z|^{2t}e^{(t\varphi_a+(1-t)\varphi_L)_e(x)}$ 
is plurisubharmonic and a local weight of a singular Hermitian metric of $\tilde{L}$ for each $t\in [0, 1]$, it is sufficient to show that $\tilde{\varphi}$ is upper semi-continuous, 
and it follows immediately from following Lemma \ref{prop_usc}. 

Let us denote by $\psi_t$ the function 
\[
\sup\hskip-0.5em\ ^*\{\psi\colon X\to \mathbb{R}\cup\{-\infty\}\mid \psi\text{ is a }(t\varphi_A+(1-t)\varphi_L)\text{-psh function, }\psi\leq 0\}. 
\]
For proving Lemma \ref{prop_usc}, we need the following lemma. 

\begin{lemma}\label{prop_increasing}\ \\
$(1)$ The sequence $\{\frac{\psi_t}{1-t}\}_{t\in [0, 1]}$ is monotonically increasing with respect to $t$. \\
$(2)$ For all $t\in [0, 1)$, $\lim_{s\downarrow t}\frac{\psi_s}{1-s}=\frac{\psi_t}{1-t}$ holds. 
\end{lemma}

\begin{proof}
(1) Let $t\leq s$ be elements of $[0, 1]$. 
Since 
\[
s\varphi_A+(1-s)\varphi_L+\frac{1-s}{1-t}\psi_t
=\frac{1-s}{1-t}(t\varphi_A+(1-t)\varphi_L)_e+\frac{s-t}{1-t}\varphi_A
\]
holds, $\frac{1-s}{1-t}\psi_t$ is a $(s\varphi_A+(1-s)\varphi_L)$-psh function. 
As $\frac{1-s}{1-t}\psi_t\leq 0$, $\frac{1-s}{1-t}\psi_t\leq \psi_s$ holds. 

(2) According to (1), it is sufficient to show that
$\lim_{s\downarrow t}\frac{\psi_s}{1-s}\leq\frac{\psi_t}{1-t}$ holds. 
Since the sequence $\frac{s}{1-s}\varphi_A+\varphi_L+\frac{\psi_s}{1-s}$ ($=\frac{1}{1-s}(s\varphi_A+(1-s)\varphi_L)_e$) is monotonically increasing psh functions, 
the limit $\frac{t}{1-t}\varphi_A+\varphi_L+\lim_{s\downarrow t}\frac{\psi_s}{1-s}$ 
is also psh. 
As $\lim_{s\downarrow t}\frac{\psi_s}{1-s}\leq 0$, $(1-t)\lim_{s\downarrow t}\frac{\psi_s}{1-s}\leq \psi_t$ holds. 
\end{proof}

\begin{lemma}\label{prop_usc}
The function $F\colon \mathbb{C}\times U\times [0, 1]\to\mathbb{R}\cup\{-\infty\}$ defined by $F(z, x, t)=(t\varphi_A+(1-t)\varphi_L)_e(x)+t\log |z|^2$ is upper semi-continuous. 
\end{lemma}

\begin{proof}
Let us set the function $H\colon U\times [0, 1]\to \mathbb{R}\cup\{-\infty\}$ as 
$H(x, t)=\frac{\psi_t(x)}{1-t}$. 
Since $F(z, x, t)$ is a sum of upper semi-continuous functions and $(1-t)H(x, t)$, 
it is sufficient to show that $H$ is upper semi-continuous. 
Let us fix an element $(x_0, t_0)\in U\times [0, 1)$ and sufficiently small positive number $\varepsilon$. 
Then, by Lemma \ref{prop_increasing} (1), 
\[
\limsup_{(x, t)\mapsto(x_0, t_0)}H(x, t)
=\lim_{r\downarrow 0}\sup_{\stackrel{|x-x_0|<r}{|t-t_0|<r}}H(x, t) 
\leq \lim_{r\downarrow 0}\sup_{|x-x_0|<r}H(x, t_0+\varepsilon) 
\]
holds. As $H(-, t_0+\varepsilon)=\frac{\psi_{t_0+\varepsilon}}{1-(t_0+\varepsilon)}$ is upper semi-continuous, we obtain an inequality $\limsup_{(x, t)\mapsto(x_0, t_0)}H(x, t)\leq H(x_0, t_0+\varepsilon)$. 
By Lemma \ref{prop_increasing} (2), we can show that the equality $\limsup_{(x, t)\mapsto(x_0, t_0)}H(x, t)\leq H(x_0, t_0)$ holds. 
We can also show this inequality when $t_0=1$ by the same argument, and this shows the lemma. 
\end{proof}

\subsection{Proof of Proposition \ref{lem_ms}}
Next we prove Proposition \ref{lem_ms}. 
Let us fix a (sufficiently positive) K\"ahler metric $\omega$ of $X$ and define 
\[
\tilde{\omega}=\pi^*\omega+dd^c\log(|z|^2e^{\varphi_A}+e^{\varphi_L})-\frac{|z|^2e^{\varphi_A}\pi^*dd^c\varphi_A+e^{\varphi_L}\pi^*dd^c\varphi_L}{|z|^2e^{\varphi_A}+e^{\varphi_L}}, 
\]
where $dd^c=\frac{\sqrt{-1}}{2\pi}\partial\overline{\partial}$. 
This $\tilde{\omega}$ defines a global smooth $(1, 1)$-form on $\tilde{X}$, 
since $dd^c\log(|z|^2e^{\varphi_A}+e^{\varphi_L})$ is the curvature form of the smooth Hermitian metric of  $\tilde{L}$ associated to 
the Finsler metric of $A\oplus L$ induced from $h_A$ and $h_L$, and both the coefficients ${|z|^2e^{\varphi_A}}/{(|z|^2e^{\varphi_A}+e^{\varphi_L})}$ of $\pi^*dd^c\varphi_A$ and ${e^{\varphi_L}}/{(|z|^2e^{\varphi_A}+e^{\varphi_L})}$ of $\pi^*dd^c\varphi_L$ glue up to define $\mathbb{R}$-valued functions defined on whole $\tilde{X}$. 

\begin{lemma}\label{prop_omega}
The form $\tilde{\omega}$ and the measures $dV_\omega=\frac{\omega^{n}}{n!}$ of $X$ and $dV_{\tilde{\omega}}=\frac{\tilde{\omega}^{n+1}}{(n+1)!}$ of $\tilde{X}$ satisfy the following properties when $\omega$ is sufficiently positive. 
\begin{enumerate}
\item $\tilde{\omega}$ is a  smooth strictly positive $(1, 1)$-form on $\tilde{X}$. 
\item For all $x\in X$, $\int_{z\in\pi^{-1}(x)}dV_{\tilde{\omega}}|_{\pi^{-1}(x)}=1$ holds. 
\item For all $\mathbb{R}$-valued measurable function $F$ on $\tilde{X}$, the equation 
\[
\int_{(z, x)\in\tilde{X}}F(z, x)\,dV_{\tilde{\omega}}=\int_{x\in X}\left(\int_{z\in\pi^{-1}(x)}F(z, x)\,dV_{\tilde{\omega}}|_{\pi^{-1}(x)}\right)\,dV_\omega
\]
holds. 
\item Moreover, when $F$ depends only on $x$ and $|z|$, an equation
\[
\int_{(z, x)\in\tilde{X}}F(z, x)\,dV_{\tilde{\omega}}=\int_{x\in X}\left(\int_0^\infty\frac{2rG(r, x)e^{\varphi_A(x)+\varphi_L(x)}}{(r^2e^{\varphi_A(x)}+e^{\varphi_L(x)})^2}\,dr\right)\,dV_\omega
\]
holds, where $G$ is the function such that $G(|z|, x)=F(z, x)$ holds. 
\end{enumerate}
\end{lemma}

\begin{proof}
By straight forward computations, we can obtain the formula 
\[
\tilde{\omega}=\pi^*\omega+C(\eta\wedge\overline{\eta}+dz\wedge\overline{\eta}+\eta\wedge d\overline{z}+dz\wedge d\overline{z}), 
\]
where $C=\frac{\sqrt{-1}}{2\pi}\frac{e^{\varphi_A+\varphi_L}}{(|z|^2e^{\varphi_A}+e^{\varphi_L})^2}$ and $\eta=z\partial(\varphi_A-\varphi_L)$. 
From this formula, it is shown that $\tilde{\omega}^{n+1}=(n+1)Cdz\wedge d\overline{z}\wedge(\pi^*\omega)^n$ holds, which shows the lemma. 
\end{proof}

We also use the following lemma, which can be proved by straight forward computations. 

\begin{lemma}\label{prop_int}
\[
-\log \int_{z\in\pi^{-1}(x)}|z|^{2t}e^{-\tilde{\varphi}_\infty(z, x)}dV_{\tilde{\omega}}|_{\pi^{-1}(x)}=t\varphi_A(x)+(1-t)\varphi_L(x)-\log\frac{\Gamma(1+t)\Gamma(2-t)}{4}
\]
holds for all $t\in [0, 1]$ and $x\in X$, where $\Gamma$ stands for the Gamma function. 
\end{lemma}

The following lemma can be shown by using the approximation theorem \cite[13.21]{D}. 

\begin{lemma}\label{prop_comp}
Let $Y$ be a smooth projective variety, $dV_Y$ a smooth volume form of $Y$, $M$ a pseudo-effective line bundle over $Y$, and let $h_M=e^{-\psi_\infty}$ be a smooth Hermitian metric of $M$. 
Fix points $y_0, y_1, \dots y_N\in Y$ and local coordinates systems around each $y_j$ such that 
$
\bigcup_{j=0}^N\{y\mid |y-y_j|<\frac{1}{\sqrt{\pi}}\}=Y
$ 
holds. 
Let $h_{M, 1}=e^{-\psi_1}$ and $h_{M, 2}=e^{-\psi_2}$ be singular Hermitian metrics given by 
\begin{eqnarray}
\psi_1 &=& \psi_\infty+ \sup\hskip-0.5em\ ^*\left\{\left. \frac{1}{m}\log |f|_{h_M^m}^2\right| m\in\mathbb{N}, f\in H^0(Y, mM), \log |f|_{h_M^m}^2\leq 0\right\}, \nonumber\\
\psi_2 &=& \psi_\infty+ \sup\hskip-0.5em\ ^*\left\{\left. \frac{1}{m}\log |f|_{h_M^m}^2\right| m\in\mathbb{N}, f\in H^0(Y, mM), \int_Y |f|_{h_M^m}^2dV_Y\leq 1\right\}. \nonumber
\end{eqnarray}
Let $C'=C'_1+C'_2$ with
$C'_1=\max_j\left(\max_{|y-y_j|\leq \frac{2}{\sqrt{\pi}}}\psi_\infty(y)-\min_{|y-y_j|\leq \frac{2}{\sqrt{\pi}}}\psi_\infty(y)\right)$ 
and $C'_2=\log\max_j\max_{|y-y_j|\leq \frac{2}{\sqrt{\pi}}}\frac{d\lambda}{dV_Y}$, 
where $d\lambda$ is the Euclidean measure. 
Then an inequality $\psi_2-C'\leq \psi_1\leq (\psi_\infty)_e$ holds. 
Moreover, if $M$ is big, then an inequality  $\psi_2-C'\leq \psi_1\leq (\psi_\infty)_e\leq \psi_2$ holds. 
\end{lemma}

\begin{proof}[Proof of Proposition \ref{lem_ms}]
(1) We fix points $x_0, x_1, \dots x_N\in X$ and local coordinates systems around each $x_j$ such that 
$
\bigcup_{j=0}^N\{x\mid |x-x_j|<\frac{1}{\sqrt{\pi}}\}=X
$ 
holds. 
Let us denote by $\varphi_{\infty, t}$ the weight of the ``smooth Hermitian metric" $t\varphi_A+(1-t)\varphi_L-\log\frac{\Gamma(1+t)\Gamma(2-t)}{4}$ of $tA+(1-t)L$. 
We let $C=C_1+C_2+\log 2$ with 
\begin{eqnarray}
C_1&=&\max_j\left(\max_{\stackrel{|x-x_j|\leq \frac{2}{\sqrt{\pi}}}{t\in [0, 1]}}\tilde{\varphi}_{\infty, t}(x)-\min_{\stackrel{|x-x_j|\leq \frac{2}{\sqrt{\pi}}}{t\in [0, 1]}}\tilde{\varphi}_{\infty, t}(x)\right), \nonumber \\
C_2&=&\log\max_j\max_{|x-x_j|\leq \frac{2}{\sqrt{\pi}}}\frac{d\lambda}{dV_\omega}. \nonumber
\end{eqnarray}
Since $(\varphi_{\infty, t})_e=(t\varphi_A+(1-t)\varphi_L)_e-\log\frac{\Gamma(1+t)\Gamma(2-t)}{4}$ holds, 
it is sufficient to show that 
\[
(\tilde{\varphi}_\infty)_e\leq\log \max_{t\in [0, 1]}|z|^{2t}e^{(\varphi_{\infty, t})_e(x)}+C 
\]
holds. 
According to the last part of Lemma \ref{prop_comp}, this is reduced to show that 
for each $F\in H^0(\tilde{X}, m\tilde{L})$ such that $\int_{\tilde{X}}|F|^2e^{-m\tilde{\varphi}_\infty}\,dV_{\tilde{\omega}}\leq 1$, 
an inequality 
\[
\frac{1}{m}\log |F|^2\leq \log \max_{t\in [0, 1]}|z|^{2t}e^{(\varphi_{\infty, t})_e(x)}+C
\]
holds. 

(2) We show the last inequality. 
The holomorphic section $F(z, x)$ can be expanded as $F(z, x)=\sum_{\ell=0}^mz^\ell f_\ell(x)$ with $f_\ell\in H^0(X, \ell A+(m-\ell)L)$. 
We first show that an inequality 
\begin{align}
\int_{\tilde{X}}|z^\ell f_\ell|^2e^{-m\tilde{\varphi}_\infty}\,dV_{\tilde{\omega}}&\leq 1\tag{$*$}
\end{align}
holds for $\ell=1, 2, \dots, m$. 
For proving ($*$), we use an inequality
\[
|f_\ell(x)|^2=\left|\frac{1}{\ell!}\frac{\partial^\ell}{\partial z^\ell}F(0, x)\right|^2\leq \frac{1}{2\pi}\int_0^{2\pi}\frac{|F(re^{\sqrt{-1}\theta}, x)|^2}{r^{2\ell}}\,d\theta 
\]
for each positive number $r$. 
We denote by $\tilde{\varphi}_\infty(r, x)$ the function such that $\tilde{\varphi}_\infty(|z|, x)=\tilde{\varphi}_\infty(z, x)$ holds. 
By multiplying the metric terms and integrating these with $r$, we obtain the following inequality. 
\begin{eqnarray}
&&\int_0^\infty\frac{2r(r^{2\ell}|f_\ell(x)|^2e^{-m\tilde{\varphi}_\infty(r, x)})e^{\varphi_A(x)+\varphi_L(x)}}{(r^2e^{\varphi_A(x)}+e^{\varphi_L(x)})^2}\,dr \nonumber\\
&\leq& \frac{1}{2\pi}\int_0^\infty\left(\int_0^{2\pi}\frac{2r(|F(re^{\sqrt{-1}\theta}, x)|^2e^{-m\tilde{\varphi}_\infty(r, x)})e^{\varphi_A(x)+\varphi_L(x)}}{(r^2e^{\varphi_A(x)}+e^{\varphi_L(x)})^2}\,d\theta\right) dr \nonumber \\
&=& \int_{z\in\pi^{-1}(x)}|F(z, x)|^2e^{-m\tilde{\varphi}(z, x)}\,dV_{\tilde{\omega}}|_{\pi^{-1}(x)}.  \nonumber
\end{eqnarray}
This inequality and Lemma \ref{prop_omega} (2), (3), (4) implies the inequality ($*$). 

Then, by Lemma \ref{prop_omega} (3), 
\begin{eqnarray}
1&\geq&\int_{(z, x)\in\tilde{X}}|z^\ell f_\ell(x)|^2e^{-m\tilde{\varphi}_\infty(z, x)}\,dV_{\tilde{\omega}} \nonumber \\
&=&\int_{x\in X}|f_\ell(x)|^2\left(\int_{z\in\pi^{-1}(x)}|z^\ell|^2e^{-m\tilde{\varphi}_\infty(z, x)}\,dV_{\tilde{\omega}}|_{\pi^{-1}(x)}\right)\,dV_\omega \nonumber \\
&=&\int_{x\in X}|f_\ell(x)|^2\left(\left(\int_{z\in\pi^{-1}(x)}\left(|z|^{2\frac{\ell}{m}}e^{-\tilde{\varphi}_\infty(z, x)}\right)^m\,dV_{\tilde{\omega}}|_{\pi^{-1}(x)}\right)^{\frac{1}{m}}\right. \nonumber \\
& &\hskip50mm\left. \cdot\left(\int_{z\in\pi^{-1}(x)}1^{\frac{m}{m-1}}\,dV_{\tilde{\omega}}|_{\pi^{-1}(x)}\right)^{\frac{m-1}{m}}\right)^m\,dV_\omega \nonumber \\
&\geq&\int_{x\in X}|f_\ell(x)|^2\left(\int_{z\in\pi^{-1}(x)}|z|^{2\frac{\ell}{m}}e^{-\tilde{\varphi}_\infty(z, x)}\right)^m\,dV_\omega \nonumber  
\end{eqnarray}
holds. Therefore, by Lemma \ref{prop_int}, $\int_{x\in X}|f_\ell(x)|^2e^{-m\varphi_{\infty, t}(x)}\,dV_\omega\leq 1$ holds for $t=\frac{\ell}{m}$. Then by Lemma \ref{prop_comp}, we obtain an inequality 
$\frac{1}{m}\log |f_\ell|^2\leq (\varphi_{\infty, t})_e+C_1+C_2$. 
Thus 
\begin{eqnarray}
&&\frac{1}{m}\log |F(z, x)|^2 \nonumber \\
&\leq&\frac{1}{m}\log \sum_{l=0}^m|z^\ell f_\ell(x)|^2 \nonumber \\
&\leq&\frac{1}{m}\log \left((m+1)\max_\ell|z^\ell f_\ell(x)|^2\right) \nonumber \\
&=&\frac{1}{m}\log (m+1)+\log\max_{0\leq \ell\leq m}|z|^{2\frac{\ell}{m}}|f_\ell(x)|^{\frac{2}{m}} \nonumber \\
&\leq&\log 2+\log\max_{t\in [0, 1]}|z|^{2t}e^{(\varphi_{\infty, t})_e(x)+C_1+C_2} \nonumber \\
&=&\log\max_{t\in [0, 1]}|z|^{2t}e^{(\varphi_{\infty, t})_e(x)}+C \nonumber 
\end{eqnarray}
holds. 
\end{proof}

\section{Proof of Theorem \ref{main_theorem}}

In this section, we prove Theorem \ref{main_theorem}. 
We first prove Theorem \ref{P1_corollary}, a special case of Theorem \ref{main_theorem}. 

\begin{proof}[Proof of Theorem \ref{P1_corollary}]
In order to show Theorem \ref{P1_corollary}, it is sufficient to consider the singular Hermitian metric with local weight function $\log (|z|^2e^{\varphi_A(x)}+e^{\varphi_L(x)})$ 
where $(z, x)$ is the coordinates just as in Theorem \ref{P1_main_theorem} and $\varphi_A, \varphi_L$ is the local weight of a minimal singular metric of $A, L$, respectively.  
Though this metric does not have minimal singularity in general, it is an extension of the minimal singular metric of $\mathcal{O}_{\mathbb{P}(A\oplus L)}(1)|_{\mathbb{P}(L)}$. 
\end{proof}

\begin{proof}[Proof of Theorem \ref{main_theorem}]
Let $X$ be a smooth projective variety, 
$D$ a $1$-codimensional smooth subvariety of $X$, 
and let $L$ be a pseudo-effective line bundle over $X$. 
We assume that $A=L\otimes\mathcal{O}_X(-D)$ is semi-positive and that there is an open neighborhood $U$ of $D\subset X$ biholomorphic to an open neighborhood $U'$ of the zero section of the normal bundle $N_{D/X}$. 
Here we may assume that $U'=\{\xi\in N_{D/X}\mid |\xi |_{h_{X/D}}<\varepsilon_0\}$ for some smooth Hermitian metric $h_{X/D}$ with negative curvature of $N_{D/X}$ and a positive number $\varepsilon_0$. 

Since $L|_D$ has no singular Hermitian metric of with psh local weights (which is not identically equal to $-\infty$) when $L|_D$ is not pseudo-effective, all we have to do is showing the existence of a singular Hermitian metric of $L$ with psh local weights which is an extension of a minimal singular metric of $L|_D$ assuming $L|_D$ is pseudo-effective but not big. 
We set $X'$ as the total space $\pi\colon\mathbb{P}(L|_D\oplus A|_D)\to D$ and $L'$ as the relative hyperplane bundle $\mathcal{O}_{\mathbb{P}(L|_D\oplus A|_D)}(1)$. 
Let us fix a minimal singular metric $h_{L'}=e^{-\varphi_{L'}}$ of $L'$. 
We set $V'$ as the subset $\{\xi\in N_{D/X}\mid |\xi |_{h_{X/D}}<\frac{\varepsilon_0}{2}\}$. 
By Remark \ref{P1bdl_rmk}, we can regard $U'$ and $V'$ be neighborhoods of $D'=\mathbb{P}(L|_D)\subset X'$. 
We denote by $V$ the set $f(V')\subset U$, where $f\colon U'\to U$ is the biholomorphic mapping. 

By Proposition \ref{rmk_metric} below, there exists a line bundle $F$ on $U'$ which admits a flat structure and $f^*(L|_U)\cong L'|_{U'}\otimes F$ holds. 
We fix a flat metric $h_F=e^{-\varphi_F}$ of $F$. 
By choosing appropriate local trivialization, we may assume $\varphi_F\equiv 0$. 
Thus we can regard $(f^{-1})^*\varphi_{L'}$ as the local weight function of the singular Hermitian metric $(f^{-1})^*h_{L'}h_F$ of $L|_{U}$. 
To show the theorem, according to Theorem \ref{P1_corollary}, it is sufficient to construct a singular Hermitian metric $e^{-\varphi_L}$ of $L$ with $dd^c\varphi_L\geq 0$ and $\varphi_L|_{V}=(f^{-1})^*\varphi_{L'}|_{V'}$ holds. 
Let $h_A=e^{-\varphi_A}$ be a smooth Hermitian metric of $A$ with $dd^c\varphi_A\geq 0$ and let $f_D\in H^0(X, \mathcal{O}_X(D))$ be a section which vanishes only on $D$. 
Without loss of generality, we may assume $\varphi_A\geq 0, (f^{-1})^*\varphi_{L'}\leq -1$ holds on each fixed open set $W_j (j=1, 2, \dots, N)$ covering the whole $U$, and $\log |f_D|^2\geq-1$ holds on each intersection $W_j\cap(\overline{U}\setminus V)$. 
We define $\varphi_L$ as the function $\max\{\varphi_A+\log |f_D|^2, (f^{-1})^*\varphi_{L'}\}$ on each $W_j\cap U$. 
Since $\varphi_L=\varphi_A+\log |f_D|^2$ holds on each intersection $W_j\cap(\overline{U}\setminus V)$, 
$e^{-\varphi_L}$ on $U$ and $e^{-(\varphi_A+\log |f_D|^2)}$ on $X\setminus V$ glue up to define a new singular Hermitian metric of $L$, which proves the theorem. 
\end{proof}

\begin{proposition}\label{rmk_metric}
There is a line bundle $E$ on $D'$ such that $c_1(E)=0$ and 
$f^*(L|_U)\cong (L'\otimes \pi^*E)|_{U'}$ hold. 
\end{proposition}

In order to prove this proposition, we use the following form of Rossi's theorem \cite[Theorem 3]{R}. 

\begin{lemma}[a version of Rossi's theorem]\label{rossi}
The natural map $H^1(U', \mathcal{O}_{U'})\to H^1(U', \mathcal{O}_{U'}/I_{D'}^n)$ is injective for some $n\geq 1$, where $I_{D'}$ the defining ideal sheaf of $D\subset U$. 
\end{lemma}

We first use Lemma \ref{rossi} and show Proposition \ref{rmk_metric}. 

\begin{proof}[Proof of Proposition \ref{rmk_metric}]
The projection $\pi\colon U'\to D$ and the injection $i\colon D'\to U'$ induce the maps 
$\pi^*\colon H^1(D', \mathcal{O}_{D'})\to H^1(U', \mathcal{O}_{U'})$ and 
$i^*\colon H^1(U', \mathcal{O}_{U'})\to H^1(D', \mathcal{O}_{D'})$, respectively. 
Since $\pi\circ i=\text{id}_{D'}$, $\pi^*$ is injective. 

\[\xymatrix{
 H^1(U', \mathcal{O}_{U'}) \ar[r]^\alpha \ar@{}[dr]|\circlearrowleft  &  H^1(U', \mathcal{O}_{U'}^*) \ar[r]^\delta  &  H^2(U', \mathbb{Z})\\
H^1(D', \mathcal{O}_{D'}) \ar[r]^\beta \ar[u]^{\pi^*} & H^1(D', \mathcal{O}_{D'}^*) \ar[u]^{\pi^*} &  \\
}\]

We first check that $f^*(L|_U)\otimes L'|_{U'}^{-1}$ is topologically trivial line bundle. 
Indeed, $(f^*(L|_U)\otimes L'|_{U'}^{-1})|_{D'}$ is the trivial bundle and $i\circ \pi$ is homotopic to $\text{id}_{U'}$. Thus we conclude that $\delta(f^*(L|_U)\otimes L'|_{U'}^{-1})=0$ and we can take an element $\xi \in H^1(U', \mathcal{O}_{U'})$ satisfying $\alpha(\xi)=f^*(L|_U)\otimes L'|_{U'}^{-1}$. When $\xi$ lies in the image of $\pi^*$, we can take an element $\eta\in H^1(D', \mathcal{O}_{D'})$ such that $\pi^*(\eta)=\xi$ holds. 
In this case, $f^*(L|_U)\otimes L'|_{U'}^{-1}=\pi^*\beta(\eta)$ holds and since $\beta(\eta)$ is a flat line bundle, $f^*(L|_U)\otimes L'|_{U'}^{-1}$ is also a flat line bundle. 

Thus all we have to do is showing that the inequality $\text{dim}\,H^1(U', \mathcal{O}_{U'})\leq \text{dim}\,H^1(D', \mathcal{O}_{D'})$ holds. 
Let us consider the short exact sequence $0\to I_{D'}^l/I_{D'}^{l+1}\to \mathcal{O}_{U'}/I_{D'}^{l+1}\to\mathcal{O}_{U'}/I_{D'}^l\to 0$ for $l\geq 1$. 
Then it follows that the natural map $H^1(U', \mathcal{O}_{U'}/I_{D'}^{l+1})\to H^1(U', \mathcal{O}_{U'}/I_{D'}^l)$ is injective. 
It is because $H^1(U',  I_{D'}^l/I_{D'}^{l+1})=H^1(U',  I_{D'}^l\otimes(\mathcal{O}_{U'}/I_{D'}))=H^1(D', \mathcal{O}_{D'}(-lD'|_{D'}))$ vanishes for each $l\geq 1$, since $\mathcal{O}_{D'}(-K_{D'}-lD'|_{D'})=\mathcal{O}_{D'}(-K_{D'}-D'|_{D'})\otimes\mathcal{O}_{D'}(-(l-1)D'|_{D'})$ is nef and big from the assumption. 
From this combined with the injection in Lemma \ref{rossi}, 
it holds that the natural map $H^1(U', \mathcal{O}_{U'})\to H^1(U', \mathcal{O}_{U'}/I_{D'})=H^1(D', \mathcal{O}_{D'})$ is injective, 
and thus we obtain the inequality $\text{dim}\,H^1(U', \mathcal{O}_{U'})\leq \text{dim}\,H^1(D', \mathcal{O}_{D'})$. 
\end{proof}

\begin{remark}
When $\mathcal{O}(K_D)$ is semi-negative, we can prove $H^1(U', \mathcal{O}_{U'})\cong H^1(D', \mathcal{O}_{D'})$ more shortly. 
Let us consider the short exact sequence $0\to I_{D'}\to \mathcal{O}_{U'}\to \mathcal{O}_{U'}/I_{D'}\to 0$ and  the induced exact sequence $H^1(U', I_{D'})\to H^1(U', \mathcal{O}_{U'})\to H^1(D', \mathcal{O}_{D'})\to H^2(U', I_{D'})$. 
By the assumption that $\mathcal{O}_D(-K_D)=\mathcal{O}_U(-K_{U}-D)|_D$ is semi-positive and by Ohsawa's theorem \cite[4.5]{O}, it follows that the cohomology group $H^p(U', I_{D'})$ vanishes for all $p>0$. 
Thus $H^1(U', \mathcal{O}_{U'})\cong H^1(D', \mathcal{O}_{D'})$ holds. 
\end{remark}

\begin{proof}[Proof of Lemma \ref{rossi}]
We intrinsically use Rossi's theorem \cite[Theorem 3]{R}. 
Here we remark that, from the assumption that $\mathcal{O}(-D)$ is ample, $U'$ is a strongly pseudoconvex domain. 
Thus, from Rossi's theorem, it turns out that there exists an ideal sheaf $J\subset \mathcal{O}_{U'}$ satisfying the condition that 
(i) $V(J)\subset D'\cup \{p_1, p_2, \cdots. p_l\}$ for some finitely many points $p_1, p_2, \cdots, p_l\in U'\setminus D'$, where $V(J)\subset U'$ stands for the zero set of the ideal sheaf $J$, and that 
(ii) the natural map $H^1(U', \mathcal{O}_{U'})\to H^1(U', \mathcal{O}_{U'}/J)$ is injective. 
Here we remark that $H^1(U', \mathcal{O}_{U'}/J)=H^1(D', \mathcal{O}_{U'}/J)$ holds. 
It is because the condition (i) and the fact that the first sheaf cohomology vanishes on the zero-dimensional sets $p_1, p_2, \cdots, p_l$. 

Let us denote by $I_{D'}$ the defining ideal sheaf of $D'$, 
by $I_{p_j}$ the defining ideal sheaf of $p_j$ for $1\leq j\leq l$, 
and by $\hat{J}$ the ideal sheaf $I_{p_1}I_{p_2}\cdots I_{p_l}I_{D'}$. 
By Hilbert's Nullstellensatz, there exists an integer $n$ such that $\hat{J}^n\subset J$ holds. 
Thus the natural map $H^1(U', \mathcal{O}_{U'})\to H^1(U', \mathcal{O}_{U'}/J)$ is decomposed into the composition 
of two natural maps $H^1(U', \mathcal{O}_{U'})\to H^1(U', \mathcal{O}_{U'}/\hat{J}^n)$ and $H^1(U', \mathcal{O}_{U'}/\hat{J}^n)\to H^1(U', \mathcal{O}_{U'}/J)$. 
From the condition (ii), it turns out that the map $H^1(U', \mathcal{O}_{U'})\to H^1(U', \mathcal{O}_{U'}/\hat{J}^n)$ is also injective, 
and since $H^1(U', \mathcal{O}_{U'}/\hat{J}^n)=H^1(D', \mathcal{O}_{U'}/\hat{J}^n)=H^1(D', \mathcal{O}_{U'}/I_{D'}^n)$ holds, this proves the lemma. 
\end{proof}

\begin{remark}\label{main_rmk}
In the above proof of Theorem \ref{main_theorem}, we compared the singular Hermitian metric of $L$ with that of $L'$ around the tubular neighborhoods of the divisors. 
By using this technique, it turns out to be clear that the metric $e^{-\varphi_L}$ we constructed above is a minimal singular metric. 
Moreover, $\varphi_{L'}$ in the above proof of Theorem \ref{main_theorem} can be taken as in Theorem \ref{P1_main_theorem} when $A|_D$ is ample, 
and thus we can conclude that the minimal singular metric we constructed has just the same form as the metric in Theorem \ref{P1_main_theorem} around $D$ (up to smooth harmonic function). 
This means that we here determined a minimal singular metric of $L$ around $D$ by only using equilibrium metrics of $tA|_D+(1-t)L|_D$ for $0\leq t\leq 1$ in the above proof in this case. 
\end{remark}

When $L$ in Theorem \ref{main_theorem} satisfies that $L|_D$ is semi-positive, we can say that $L$ is also semi-positive. 

\begin{corollary}\label{semipositive}
Let $X, D, L$ be those in Theorem \ref{main_theorem}. 
When $L|_D$ is semi-positive, $L$ is also semi-positive. 
\end{corollary}

\begin{proof}
We use notations in the proof of \ref{main_theorem}. 
By the proof of Theorem \ref{P1_corollary}, it is clear that we can choose smooth $h_{L'}$ when $L|_D$ is semi-positive. 
We define $\varphi_L$ as the function $M(\varphi_A+\log |f_D|^2, (f^{-1})^*\varphi_{L'})$ (instead of $\max\{\varphi_A+\log |f_D|^2, (f^{-1})^*\varphi_{L'}\})$ on each $W_j\cap U$, where $M$ is a regularized max function (see \cite[\S 5.E]{D1} for the definition). 
Then $\{e^{-\varphi_L}\}$ glues up to define a smooth Hermitian metric of $L$ with semi-positive curvature. 
\end{proof}

We here remark that the idea to use a regularized max function instead of the function ``$\max$" is pointed out by Prof. Shin-ichi Matsumura. 

\section{Some examples}

\subsection{Nef and big line bundles with no locally bounded minimal singular metrics}

One can obtain the following corollary immediately from Theorem \ref{P1_corollary}. 

\begin{corollary}\label{nefbig}
Let $X$ be a smooth projective variety, $L$ a nef line bundle over $X$ and let $A$ be an ample line bundle over $X$. 
Then a minimal singular metric of $L$ is locally bounded 
if and only if a minimal singular metric of $\mathcal{O}_{\mathbb{P}(A\oplus L)}(1)$ over $\mathbb{P}(A\oplus L)$ is locally bounded. 
\end{corollary}

We remark that the line bundle $\mathcal{O}_{\mathbb{P}(A\oplus L)}(1)$ above is nef and big (\cite[2.3.2]{L1}). 

\begin{example}\label{dps_example}
Let $(X, L)$ be these in Example 1.7 of \cite{DPS94}, which are defined as the relative hyperplane bundle on $X=\mathbb{P}(E)$, where $E$ is a vector bundle defined over an elliptic curve $C$ given by the non-spitting extension 
$
0\to\mathcal{O}_C\to E\to\mathcal{O}_C\to 0. 
$
In this example, $L$ is nef, not big, and possesses no locally-bounded minimal singular metric. 
Then we can conclude that the nef and big line bundle $\mathcal{O}_{\mathbb{P}(A\oplus L)}(1)$ defined on $\mathbb{P}(L\oplus A)$ for some ample line bundle $A$ on $X$ also has no locally-bounded minimal singular metric. 
We remark that the similar example is introduced in \cite[5.4]{BEGZ}, \cite[5.2]{F}. 
\end{example}

\subsection{Zariski's and Mumford's examples}

We can apply Theorem \ref{main_theorem} to Zariski's and Mumford's examples \cite[2.3.A]{L1}. 

\begin{example}\label{Zariski}
Let $C\subset \mathbb{P}^2$ be a smooth elliptic curve and let $p_1, p_2, \dots, p_{12}\in C$ be twelve general points. 
We define $X$ as the blow up of $\mathbb{P}^2$ at these twelve points. 
We denote by $H$ the pulled back divisor of $X$ of a line in $\mathbb{P}^2$ and by $D$ the strict transform of $C$. 
In this case, since $(D^2)=9-12=-3$ and the genus $g(D)=1$, 
we can apply Grauert's theorem \cite[Satz 7]{G} (see \S 1 here) to see that $X, L=\mathcal{O}_X(H+D)$, and $D$ satisfy the condition of Theorem \ref{main_theorem}. 
Moreover, for $L|_D$ is semi-positive, we can apply Corollary \ref{semipositive}. 
Thus $L$ is semi-positive. 

There is a generalization of this Zariski's example pointed out by Mumford (see also \cite[2.3.1]{L1}). 
Let $X$ be a smooth projective surface, $A$ a very ample divisor on $X$, and let $D\subset X$ be a curve with $(D^2)<0$ holds and the restriction map ${\rm Pic}(X)\to{\rm Pic}(D)$ is injective. 
We denote by $a, b$ the positive number $(A. D), -(D^2)$, respectively.  
Then the line bundle $L=\mathcal{O}_X(bA+aD)$ is nef, big, satisfying $D\subset\text{Bs}\,|L^{\otimes m}|$ for all $m\geq 1$, and there exists a positive integer $p_0$ such that $|L^{\otimes m}\otimes\mathcal{O}_X(-p_0D)|$ is generated by global sections for all $m\geq 1$. 
These $X, L$, and $D$ satisfy the condition of Corollary \ref{semipositive} also in this situation when $D$ is smooth, $b$ is sufficiently large, and $p_0=1$. 
Thus, such $L$ is semi-positive, too. 
\end{example}


\end{document}